\documentclass[11pt,a4paper]{amsart}
\usepackage{amssymb}
\usepackage{amsfonts}
\usepackage{}
\usepackage{amsfonts}
\usepackage{mathrsfs}

\usepackage[bookmarks,colorlinks,
            linkcolor=black,
            anchorcolor=black,
            citecolor=blue]{hyperref}
\usepackage{color}              
\usepackage{indentfirst}        
\usepackage{latexsym}
\usepackage{amsmath,amssymb}    
\usepackage{pstricks}
\usepackage{pst-node}
\usepackage{pst-tree}
\usepackage{pst-plot}
\usepackage{pst-text}
\usepackage{graphicx}
\usepackage{cases}
\usepackage{pifont}
\usepackage{txfonts}
\usepackage[all,knot,poly]{xy}

\usepackage{arydshln}

\allowdisplaybreaks[4]  
\numberwithin{equation}{section}

\newcommand{\ga}{\gamma}
\newcommand{\Ga}{\Gamma}
\newcommand{\la}{\lambda}

\newcommand{\fs}{\mathfrak{S}}

\newcommand{\bN}{\mathbb{N}}
\newcommand{\bQ}{\mathbb{Q}}

\newcommand{\si}{\sigma}

\newcommand{\mb}[1]{\mbox{#1}}

\newcommand{\lb}{\left(}
\newcommand{\rb}{\right)}

\newcommand{\beq}{\begin{equation}}
\newcommand{\eeq}{\end{equation}}

\theoremstyle{plain}

\newtheorem{theorem}{Theorem}[section]

\newtheorem{lem}[theorem]{Lemma}

\newtheorem{prop}[theorem]{Proposition}

\theoremstyle{remark}
\newtheorem{rem}[theorem]{Remark}

\theoremstyle{definition}

\newtheorem{exam}[theorem]{Example}

\begin{document}

\title[The distance spectra of the derangement graphs]
{The distance spectra of the derangement graphs}

\author[Li]{Yunnan Li}
\address{School of Mathematics and Information Science, Guangzhou University, Waihuan Road West 230, Guangzhou 510006, China}
\email{ynli@gzhu.edu.cn}
\author[Lin]{Huiqiu Lin$^\star$}
\address{Department of Mathematics, East China University of Science and Technology, Shanghai 200237, China}
\email{huiqiulin@126.com}

\thanks{$^\star$H. Lin, Corresponding Author.}

\begin{abstract}
In this paper, we consider the distance spectra of the derangement graphs.
First we give a constructive proof that the connected derangement graphs are of diameter 2.
Then we obtain their distance spectra. In particular, we determine all their extremal distance eigenvalues.

\end{abstract}

\maketitle
\section{introduction}
The derangement graphs $\Ga_n$ are a family of normal Cayley graphs associated with symmetric groups $\fs_n$ and their derangement sets. Recently such graphs are generalized to those related to any permutation groups and applied to the research of intersecting families of permutation groups; see \cite{CK,LM,MST}. One important algebraic invariant of a graph is its adjacency spectrum, namely the eigenvalues of its adjacency matrix. The adjacency spectra of $\Ga_n$ are deeply related to some interesting combinatorial objects, such as the factorial Schur functions \cite{CL} and shifted Schur functions \cite{OO}, etc. The first recurrence formula to calculate adjacency spectra of $\Ga_n$ is due to Renteln in \cite{Re}. Later such adjacency spectra are more subtly studied by Ku, Wales and Wong in \cite{KW,KW1}. Besides, the automorphism groups of derangement graphs are determined in \cite{DZ1}.

In this paper we mainly discuss the distance spectra of the derangement graphs. That is another significant invariant to reflect the properties of graphs. First we constructively prove that the connected derangement graphs are of diameter 2, by decomposing any non-derangement permutation into a product of two derangements. One can also obtain such fact by the intersecting properties of permutations given in \cite{CK}. As a result, we can successively derive the distance spectra of $\Ga_n$ based on the preceding work about their adjacency spectra.
In particular, the smallest adjacency eigenvalue of $\Ga_n$ found in \cite{Re} and the second largest adjacency one given in \cite{DZ} correspond to the second largest distance eigenvalue and the smallest distance one of $\Ga_n$ respectively. Here we further determine the second smallest distance eigenvalue and the third largest one of  $\Ga_n$.

\section{Background}
For any $n\in\bN$, let $\fs_n$ be the symmetric group. The \textit{derangement graph} $\Ga_n$ of $\fs_n$ is the normal Cayley graph $\Ga(\fs_n,D_n)$ with vertices $V(\Ga_n)=\fs_n$ and edges $E(\Ga_n)=\{(w,sw):w\in\fs_n,s\in D_n\}$, where $D_n=\{w\in\fs_n:w(i)\neq i,\,i=1,\dots,n\}$. By abuse of notation, we also use $D_n$ to denote its cardinality. Permutations in $D_n$ are called \textit{derangements}. By the inclusion-exclusion principle, we know that
\beq\label{der}
D_n=\sum_{k=0}^n(-1)^k{n \choose
k}(n-k)!=n!\sum_{k=0}^n(-1)^k\dfrac{1}{k!}
=\begin{cases}
0&n=1,\\
1&n=0,2,\\
\left\{\dfrac{n!}{e}\right\}&n\geq3,
\end{cases}\eeq
where $\{k\}$ denotes the nearest integer to $k$; see also \cite[Lemma 4.1]{Re}. Approximately,
\beq\label{der1}
\left|D_n-\dfrac{n!}{e}\right|=\left|n!\sum_{k=n+1}^\infty(-1)^k\dfrac{1}{k!}\right|
<\dfrac{1}{n+1}.
\eeq
Note also the following useful induction identities.
\beq\label{ide}
\begin{array}{l}
D_n=nD_{n-1}+(-1)^n,\\
D_n=(n-1)(D_{n-1}+D_{n-2}).
\end{array}
\eeq

For any graph $\Ga$ and $u,v\in V(\Ga)$, let the \textit{distance} $d_\Ga(u,v)$ be the length of a shortest path between $u$ and $v$ in $\Ga$. The \textit{distance matrix} of $\Ga$ is defined by $d_\Ga:=(d_\Ga(u,v))_{u,v\in V(\Ga)}$, as a symmetric, nonnegative definite matrix. The \textit{distance spectrum} and the \textit{distance polynomial} $P_\Ga(q)$ of $\Ga$ are defined by the spectrum and the characteristic polynomial of $d_\Ga$ respectively.

By definition, $d_{\Ga_n}(u,v)$ is the shortest length of expressions of $vu^{-1}$ as a product of derangements for any $u,v\in\fs_n$. Hence, we define a length function $\ell_D$ on $\fs_n$ as follows. For any $w\in\fs_n$, let
\[\ell_D(w)=\mbox{min}\{r\in\bN :
w=u_1\cdots u_r,\,u_1,\dots,u_r\in D_n\}.\]
Then we compute $d_{\Ga_n}(u,v)=\ell_D(vu^{-1})$ below.

\begin{prop}\label{dis}
For any $u, v\in\fs_n\,(n>3)$, we have
\[d_{\Ga_n}(u,v)=\begin{cases}
1&vu^{-1}\in D_n,\\
2&\mbox{otherwise}.
\end{cases}\]
Equivalently, the diameter of $\Ga_n$ is two when $n>3$.
\end{prop}

\begin{proof}
First note that the length function $\ell_D$ is a conjugacy class function on $\fs_n$, since $D_n$ is stable under
conjugation.
%
Now we prove the result by induction on the number $f$ of fixed points of a permutation $w$. First we need the following identities about permutations:
\beq\label{cycle}
\begin{array}{l}
(i)(j)=(ij)\cdot(ij),\\
(i)(j)(kl)=(ik)(jl)\cdot(ikjl),
\end{array}
\eeq
 and
 \beq\label{cycle1}
\begin{array}{l}
(ij)(kl)=(ik)(jl)\cdot(il)(jk),\\
(l_1l_2\cdots l_r)=(l_1l_2\cdots l_r)^2\cdot(l_rl_{r-1}\cdots l_1),
\end{array}
\eeq
 and
\beq\label{cycle2}
\begin{array}{l}
(i)(jk)(lm)=(ilkmj)\cdot(ijlkm),\\
(i)(jk)(lm)(np)=(inj)(lkmp)\cdot(ijlknmp),\\
(i)(jk)(l_1\cdots l_r)=\begin{cases}
(ij)(l_1l_3\cdots l_rkl_2l_4\cdots l_{r-1})\cdot(ijl_r\cdots l_1k)&r\,\,\mbox{odd},\\
(ij)(l_1l_3\cdots l_{r-1})(l_2l_4\cdots l_rk)\cdot(ijl_r\cdots l_1k)&r\,\,\mbox{even}
\end{cases}
\end{array}
\eeq
for distinct tuples $i,j,k,l,m,n,p,l_1,\dots,l_r\in\bN$ with $r\geq3$.

\noindent
Case 1: $f=1$.\quad This is the most technical case. First by \eqref{cycle1}, \eqref{cycle2} one can easily see that
$\ell_D(w)=2$, when $w$ has one fixed point and at least one 2-cycle. For instance,
\[\begin{array}{l}
(1)(23)(45)=(14352)\cdot(12435),\\
(1)(23)(45)(67)=(162)(4357)\cdot(1243657),\\
(1)(23)(456)=(12)(4635)\cdot(126543),\\
(1)(23)(4567)=(12)(46)(573)\cdot(1276543).
\end{array}\]
Thus we still need to consider the case when $w$ has one fixed point and no 2-cycles.
By \eqref{cycle1} and that $\ell_D$ only depends on the
cycle type, we reduce such case to the special result:
\[\ell_D((1)(23\cdots n))=2,\,n>3.\]
For this, we find a tuple $(k_1,\dots,k_{n-2})\,(=\{3,4,\dots,n\}\mbox{ as sets})$ such that
$(1)(23\cdots n)$ is composed by the following two derangements, written as
\[\begin{array}{c@{\hspace{1em}}c@{\hspace{1em}}c@{\hspace{1em}}c@{\hspace{1em}}c@{\hspace{1em}}c}
1&2&3&\cdots&n-1&n\\
\downarrow&\downarrow&\downarrow&\cdots&\downarrow&\downarrow\\
2&k_1&k_2&\dots&k_{n-2}&1\\
\downarrow&\downarrow&\downarrow&\cdots&\downarrow&\downarrow\\
1&3&4&\cdots&n&2
\end{array}\]
Indeed we fix a positive integer $p$ such that $n-2\nmid p,p+1$.
Then take those distinct $k_j$'s in $\{3,4,\dots,n\}$ uniquely as
\[k_j\equiv j+p+2\mod n-2\]
for $j=1,\dots,n-2$.

For the permutation
$\tau=\begin{pmatrix}
1&2&3&\cdots&n-1&n\\
2&k_1&k_2&\cdots&k_{n-2}&1
\end{pmatrix}$, $\tau(j+1)=k_j$ and
$k_j-(j+1)\equiv p+1\mod n-2$ imply $\tau\in D_n$.
For the permutation
$\si=\begin{pmatrix}
2&k_1&k_2&\cdots&k_{n-2}&1\\
1&3&4&\cdots&n&2
\end{pmatrix}$,
$\si(k_j)=j+2$ and $k_j-(j+2)\equiv p\mod n-2$, thus $\si\in D_n$. Hence,
$(1)(23\cdots n)=\si\cdot\tau$ with $\si,\tau\in D_n$. For example, when $n=6$ we have
\[(1)(23456)=\begin{cases}
(12)(3654)\cdot(1246)(35)&\mbox{for }p=1,\\
(12)(35)(46)\cdot(125436)&\mbox{for }p=2.
\end{cases}\]
In summary, we have proved that $\ell_D(w)=2$ if $w$ has one fixed point.

\noindent
Case 2: $f=2$.\quad By \eqref{cycle}, \eqref{cycle1} one can easily see that
$\ell_D(w)=2$. For example,
\[\begin{array}{l}
(1)(2)(34)(56)=(12)(35)(46)(987)\cdot(12)(36)(45),\\
(1)(2)(34)(56)(78)=(13)(24)(57)(68)\cdot(1324)(58)(67),\\
(1)(2)(34567)=(12)(35746)\cdot(12)(76543).
\end{array}
\]

\noindent
Case 3: $f>2$.\quad By induction on $f$, there exist
permutations $w',w''$ on $\{3,4,\dots,n\}$ both with fixed points, such that $w=(12)w'\cdot(12)w''$.
Therefore, $\ell_D(w)=2$ if $w$ has more than two fixed points.
\end{proof}

\begin{rem}
By the intersecting properties of permutations (\cite[Prop. 6]{CK}) and the definition of $\Ga_n$, one can also see that $\Ga_n\,(n\geq4)$ is of diameter 2, while our alternative proof provides concrete construction.
\end{rem}

\section{Distance spectra of the derangement graphs}
Next we discuss the distance spectra of the derangement graphs.
For the group ring $\bQ[\fs_n]$, it is endowed with a regular representation of $\fs_n$, denoted by $\rho$. Let
$\{e_{\si_1},\dots,e_{\si_{n!}}\}_{\si_i\in\fs_n}$ be the natural basis of $\bQ[\fs_n]$.
Consider the linear operator
\[E_n:=\sum_{w\in\fs_n\backslash\{1\}}\ell_D(w)\rho(w)\]
on $\bQ[\fs_n]$. Then
\[(E_n(e_{\si_1}),\dots,E_n(e_{\si_{n!}}))=
(e_{\si_1},\dots,e_{\si_{n!}})d_{\Ga_n}.\]
Equivalently, the distance spectrum of $\Ga_n$ are the spectrum of
the operator $E$.

Since $\ell_D$ is a conjugacy class function on $\fs_n$,
$E$ commutes with any $\rho(w)\,(w\in\fs_n)$, i.e. $E$
is an endomorphism of the representation $\bQ[\fs_n]$.
By Schur's lemma, we know that the restrictions of $E$ act on irreducible representations by scalars.
Hence, we have
\begin{prop}\label{ds}
The eigenvalues of $d_{\Ga_n}$ are of the form
\[\ga_\chi=\dfrac{1}{\chi(1)}\sum_{w\in\fs_n\backslash\{1\}}\ell_D(w)\chi(w)\]
for any irreducible character $\chi$, with multiplicity $\chi(1)^2$.
\end{prop}
It is well-known that the irreducible characters $\chi$ of $\fs_n$ are indexed by partition $\la\vdash n$, thus denoted by $\chi_\la$, and $\chi_\la(1)=f_\la$, the number of standard tableaux of shape $\la$. Now abbreviate $\ga_{\chi_\la}$ as $\ga_\la$ associated with $\chi_\la$.

In fact, by Prop. \ref{dis} we know that
\[d_{\Ga_n}=A_{\Ga_n^c}+J_n=2J_n-A_{\Ga_n},\]
where $A_\Ga$ is the adjacency matrix of $\Ga$, $\Ga_n^c$ is the complement graph of $\Ga_n^c$ and
$J_n$ is the circulant matrix with the zero diagonal and other entries equal to one. Hence, one can also obtain the distance spectrum of $\Ga_n$ from its adjacency spectrum calculated in \cite{KW,Re}.

\begin{prop}\label{di}
Write
\[\left\{\eta_1=D_n,\eta_2=\dfrac{n-1}{n-3}D_{n-2},\dots,\eta_{n!}=-\dfrac{D_n}{n-1}\right\}\]
as the adjacency spectrum of $\Ga_n$, then \[\{2(n!-1)-\eta_1,-2-\eta_2,\dots,-2-\eta_{n!}\}\]
is its distance spectrum.
\end{prop}

On the other hand, the matrix of the endomorphism
$\sum\limits_{w\in\fs_n\backslash\{1\}}\rho(w)$
under the natural basis is obviously $J_n$ and $E_n=
2\sum\limits_{w\in\fs_n\backslash\{1\}}\rho(w)-\sum\limits_{w\in D_n}\rho(w)$, thus the matrix $d_{\Ga_n}$ of $E_n$ is $2J_n-A_{\Ga_n}$.
By Prop. \ref{ds}, we again have the eigenvalues of $d_{\Ga_n}$:
\[\ga_1:=2(n!-1)-\eta_1\mbox{ relative to the trivial character and }\ga_\chi:=-2-\eta_\chi,\,\chi\neq 1.\]

We recall \textit{Renteln's recurrence formula} for $\eta_\la$'s in \cite[Theorem 6.5]{Re}:
\beq\label{rec}
\eta_\la=(-1)^h(\eta_{\la-h}+(-1)^{\la_1}h\eta_{\la-1})
\eeq
with initial condition $\eta_\emptyset=1$, where $h$ denote either the principal hook of $\la$ or its cardinality,  $\la-1$ denote the partition obtained by removing the first column of $\la$.
In particular, we have
\[\begin{array}{l@{\vspace{0.3em}}}
\eta_{(n-i,1^i)}=(-1)^n+(-1)^inD_{n-1-i},\,1\leq i\leq n-1.\\
\eta_{(n-2,2)}=-(n-1)\eta_{(n-3,1)}=\dfrac{n-1}{n-3}D_{n-2}=(n-1)\left[(-1)^{n-1}+(n-2)D_{n-4}\right].\\
\eta_{(n-3,3)}=(-1)^{n-2}-(n-2)\eta_{(n-4,2)}=(-1)^{n-2}-\dfrac{(n-2)(n-3)}{n-5}D_{n-4}.
\end{array}\]
We also need the following \textit{alternating sign property} proved in \cite[Theorem 1.2]{KW}.
\beq\label{alt}\mb{sign}(\eta_\la)=(-1)^{\la-\la_1}\eeq
for any partition $\la$.


First we give several technical lemmas for the proof of our main result.
\begin{lem}\label{l2}
For $n\geq 6,\,4\leq i\leq n/2$, we have
\[|\eta_{(n-i,i)}|<|\eta_{(n+1-i,i-1)}|<\eta_{(n-2,1^2)}.\]
\end{lem}
\begin{proof}
First, we have
\[\eta_{(n-2,1^2)}=nD_{n-3}+(-1)^n\geq n\left(\dfrac{(n-3)!}{e}-\dfrac{1}{n-2}\right)-1
=\dfrac{n(n-3)!}{e}-\dfrac{2n-2}{n-2},\]
and
\[\begin{split}
|\eta_{(n-3,3)}|&=|(-1)^{n-2}D_2-(n-2)\eta_{(n-4,2)}|
\leq1+(n-2)(n-3)|\eta_{(n-5,1)}|\\
&=1+\dfrac{(n-2)(n-3)D_{n-4}}{n-5}\leq \dfrac{n-2}{n-5}\dfrac{(n-3)!}{e}+\dfrac{2n-7}{n-5}.
\end{split}\]
Hence, $|\eta_{(n-3,3)}|<\eta_{(n-2,1^2)}$ for $n\geq7$ and the case when $n=6$ is easy to see.

Now we only need to prove that $|\eta_{(n-i,i)}|<|\eta_{(n+1-i,i-1)}|$ for $4\leq i\leq n/2$. It will be done by induction on $i$. For $i=4$, it is easy to check. Suppose $i,\,n-i$ have different parity, then
\[\begin{split}
|\eta_{(n-i,i)}|&-|\eta_{(n+1-i,i-1)}|=|(-1)^{n+1-i}\eta_{(i-1)}-(n+1-i)\eta_{(n-1-i,i-1)}|\quad\mbox{by Formula \eqref{rec}}\\
&-|(-1)^{n-i}\eta_{(i-2)}-(n+2-i)\eta_{(n-i,i-2)}|\\
&=D_{i-1}-D_{i-2}+(n+1-i)\eta_{(n-1-i,i-1)}-(n+2-i)\eta_{(n-i,i-2)}\quad\mbox{by Property \eqref{alt}}\\
&<D_{i-1}-D_{i-2}-\eta_{(n-i,i-2)}<0\quad\mbox{by induction}\\
&=D_{i-1}-D_{i-2}-|(-1)^{n+1-i}\eta_{(i-3)}-(n+1-i)\eta_{(n-1-i,i-3)}|\\
&=D_{i-1}-D_{i-2}-D_{i-3}-(n+1-i)|\eta_{(n-1-i,i-3)}|\quad\mbox{by Property \eqref{alt}}\\
\end{split}\]
Thus we reduce it to the proof that $D_{i-1}-D_{i-2}-|\eta_{(n-i,i-2)}|<0$ when $i,\,n-i$ have different parity, by induction on $i$. It is clear for $i=4$. Further, we have
\[\begin{split}D_{i-1}-D_{i-2}-|\eta_{(n-i,i-2)}|
&=D_{i-1}-D_{i-2}-|(-1)^{n+1-i}\eta_{(i-3)}-(n+1-i)\eta_{(n-1-i,i-3)}|\\
&=D_{i-1}-D_{i-2}-D_{i-3}-(n+1-i)|\eta_{(n-1-i,i-3)}|\quad\mbox{by Property \eqref{alt}}\\
&<D_{i-1}-D_{i-2}-D_{i-3}-(n+1-i)(D_{i-2}-D_{i-3})\quad\mbox{by induction}\\
&=(i-3)(D_{i-2}+D_{i-3})-(n+1-i)(D_{i-2}-D_{i-3})\quad\mbox{by \eqref{ide}}\\
&=-(n+4-2i)D_{i-2}+(n-2)D_{i-3}\\
&=-(n+4-2i)((i-2)D_{i-3}+(-1)^i)+(n-2)D_{i-3}\quad\mbox{by \eqref{ide}}\\
&=-(n+2-2i)(i-3)D_{i-3}-(-1)^i(n+4-2i)\\
&=-(n+2-2i)[(i-3)D_{i-3}+(-1)^i]-2(-1)^i\leq0\quad\mbox{when }5\leq i\leq n/2.
\end{split}\]
When $i,\,n-i$ have the same parity, then
\[\begin{split}
|\eta_{(n-i,i)}|&-|\eta_{(n+1-i,i-1)}|=|(-1)^{n+1-i}\eta_{(i-1)}-(n+1-i)\eta_{(n-1-i,i-1)}|\quad\mbox{by Formula \eqref{rec}}\\
&-|(-1)^{n-i}\eta_{(i-2)}-(n+2-i)\eta_{(n-i,i-2)}|\\
&<D_{i-1}+D_{i-2}+(n+1-i)|\eta_{(n-1-i,i-1)}|-(n+2-i)|\eta_{(n-i,i-2)}|\quad\mbox{by Property \eqref{alt}}\\
&<D_{i-1}+D_{i-2}-|\eta_{(n-i,i-2)}|\quad\mbox{by induction}.
\end{split}\]
Hence, it is reduced to the proof that $D_{i-1}+D_{i-2}-|\eta_{(n-i,i-2)}|<0$ when $i,\,n-i$ have the same parity, again by induction on $i$. It is clear for $i=4$. Moreover,
\[\begin{split}D_{i-1}+D_{i-2}-|\eta_{(n-i,i-2)}|
&=D_{i-1}+D_{i-2}-|(-1)^{n+1-i}\eta_{(i-3)}-(n+1-i)\eta_{(n-1-i,i-3)}|\\
&\leq D_{i-1}+D_{i-2}+D_{i-3}-(n+1-i)|\eta_{(n-1-i,i-3)}|\quad\mbox{by Property \eqref{alt}}\\
&<D_{i-1}+D_{i-2}+D_{i-3}-(n+1-i)(D_{i-2}+D_{i-3})\quad\mbox{by induction}\\
&=D_{i-1}-\dfrac{n-i}{i-2}D_{i-1}<0\quad\mbox{by \eqref{ide} and as }4\leq i\leq n/2.
\end{split}\]
In summary, we finally verify the desired inequality.
\end{proof}
We note that without using \eqref{alt} it is much more straightforward only to show that \[|\eta_{(n-i,i)}|<|\eta_{(n-3,3)}|=-\eta_{(n-3,3)},\,4\leq i\leq n/2,\]
which is enough indeed for the forthcoming discussion of our main result.


In particular, partitions $\la=(n-i,1^i)\vdash n$, $0\leq i\leq n-1$, are called \textit{hooks}.
\begin{lem}\label{l3}
For $3\leq i\leq n-1$, we have
\[|\eta_{(n-i,1^i)}|<\eta_{(n-2,1^2)}.\]
For $3\leq i\leq n-1,\,n\geq6$, we have
\[\eta_{(n-3,3)}<-|\eta_{(n-i,1^i)}|.\]
\end{lem}
\begin{proof}
The first family of inequalities are clear since \[\eta_{(n-i,1^i)}=(-1)^n\left(1+(-1)^{n-i}nD_{n-i-1}\right),\,1\leq i\leq n-1,\]
while for the second one,
\[\eta_{(n-3,3)}=(-1)^{n-2}-\dfrac{(n-2)(n-3)}{n-5}D_{n-4}<-1-nD_{n-4}\leq-|\eta_{(n-i,1^i)}|\]
for $3\leq i\leq n-1,\,n\geq6$.
\end{proof}

Partitions $\la=(n-2-i,2,1^i)\vdash n$, $0\leq i\leq n-4$, are called \textit{near hooks}.
\begin{lem}\label{l4}
For $1\leq i\leq n-4$, we have
\[|\eta_{(n-2-i,2,1^i)}|<\eta_{(n-2,1^2)}.\]
For $1\leq i\leq n-4,\,n\geq6$, we have
\[\eta_{(n-3,3)}<-|\eta_{(n-2-i,2,1^i)}|.\]
\end{lem}
\begin{proof}
Note that
\[\eta_{(n-i-2,2,1^i)}=(n-1)\left((-1)^{n-1}+(-1)^i(n-i-2)D_{n-i-4}\right),\,1\leq i\leq n-4,\]
thus we have
\[\begin{split}
|\eta_{(n-i-2,2,1^i)}|&\leq|\eta_{(n-3,2,1)}|\leq(n-1)\left(1+(n-3)D_{n-5}\right)\\
&<(-1)^n+nD_{n-3}=\eta_{(n-2,1^2)}
\end{split}\]
for $n\geq5$.
For the second one, we have
\[\begin{split}
\eta_{(n-3,3)}&=(-1)^{n-2}-\dfrac{(n-2)(n-3)}{n-5}D_{n-4}
=(-1)^{n-2}-\dfrac{(n-2)(n-3)}{n-5}\lb(n-4)D_{n-5}+(-1)^n\rb\\
&<-(n-1)((-1)^n+(n-3)D_{n-5})=-|\eta_{(n-3,2,1)}|\leq-|\eta_{(n-i-2,2,1^i)}|
\end{split}\]
for $1\leq i\leq n-4,\,n\geq6$.
\end{proof}

\begin{lem}\label{l5}
For $3\leq i\leq(n-1)/2,\,n\geq7$, we have
\[|\eta_{(n-i-1,i,1)}|<-\eta_{(n-3,3)}.\]
For $2\leq j\leq i\leq(n-2)/2$, we have
\[|\eta_{(n-i-j,i,j)}|<-\eta_{(n-3,3)}.\]
\end{lem}
\begin{proof}
By \cite[Theorem 1.1]{DZ} we know that $\eta_{(n-2,2)}$ is the second largest adjacency eigenvalue of $\Ga_n$, thus
\[\begin{split}
|\eta_{(n-i-1,i,1)}|&\leq\eta_{(i-1)}
+(n-i+1)|\eta_{(n-i-2,i-1)}|\leq D_{n-4}+(n-2)\eta_{(n-5,2)}\\
&=D_{n-4}+\dfrac{(n-2)(n-4)}{n-6}D_{n-5}
<(-1)^{n-1}+\dfrac{(n-2)(n-3)}{n-5}D_{n-4}
=-\eta_{(n-3,3)}.
\end{split}\]
For the second one, we similarly have
\[\begin{split}
|\eta_{(n-i-j,i,j)}|&\leq|\eta_{(i-1,j-1)}|
+(n-i-j+2)|\eta_{(n-i-j-1,i-1,j-1)}|\\
&\leq D_{i+j-2}+(n-i-j+2)\eta_{(n-5,2)}\leq D_{n-4}+(n-2)\eta_{(n-5,2)}<-\eta_{(n-3,3)}
\end{split}\]
when $n\geq7$, and the case for $n=6$ is also clear.
\end{proof}

\begin{rem}
In \cite[Theorem 1.6]{KW1}, Ku and Wong have proven that
for $\la,\la'\vdash n$ with $\la_1=\la'_1$, $\la\preceq\la'$ (in dominant order) implies
 $|\eta_\la|\leq|\eta_{\la'}|$. However, it is not enough for us to find all the extremal eigenvalues of $d_{\Ga_n}$.
\end{rem}

Now we are in the position to give our main results.
\begin{theorem}
The smallest eigenvalue for $d_{\Ga_n}\,(n\geq4)$ is
\[\ga_{(n-2,2)}=-2-\dfrac{n-1}{n-3}D_{n-2}=-2+(n-1)\left[(-1)^n-(n-2)D_{n-4}\right]\]
with multiplicity $\chi_{(n-2,2)}(1)^2$, and the second smallest one is
\[\ga_{(n-2,1^2)}=-2+\left[(-1)^{n-1}-nD_{n-3}\right]\]
with multiplicity $\chi_{(n-2,1^2)}(1)^2$ when $n\geq4,\,n\neq 5$. For $d_{\Ga_5}$,
$\ga_{(3,2)}=\ga_{(3,1^2)}=4$ is smallest.
\end{theorem}

\begin{proof}
First we claim $\ga_{(n-2,2)}\leq\ga_{(n-2,1^2)}$ with equality only when $n=5$. One can confirm it for the cases when $n$ is from 4 to 13 in the list in \cite[Section 11]{KW}. Now we have
\[\begin{split}
\ga_{(n-2,2)}&-\ga_{(n-2,1^2)}=n(D_{n-3}+(-1)^n)-(n-1)(n-2)D_{n-4}\\
&\leq n\left(\dfrac{(n-3)!}{e}+\dfrac{1}{n-2}+(-1)^n\right)
-(n-1)(n-2)\left(\dfrac{(n-4)!}{e}-\dfrac{1}{n-3}\right)\\
&\leq-\dfrac{2(n-4)!}{e}+\dfrac{n(n-1)}{n-2}+\dfrac{(n-1)(n-2)}{n-3}
\leq-\dfrac{2(n-4)!}{e}+2n+3,
\end{split}\]
where the first inequality is due to the bound \eqref{der1}. In particular,
it implies that $\ga_{(n-2,2)}-\ga_{(n-2,1^2)}<0$ when $n\geq9$. In summary, we verify the claim.

Now by Lemma \ref{l3} and Lemma \ref{l4}, we only need to prove that $\ga_{(n-2,1^2)}<\ga_\la$ for any partition $\la$ neither a hook nor a near hook. First for $\la=(n-i,i)\,(3\leq i\leq n/2)$, we have
$|\eta_{(n-i,i)}|<\eta_{(n-2,1^2)}$
by Lemma \ref{l2}. Hence, $\ga_{(n-2,1^2)}<\ga_{(n-i,i)}$ for $3\leq i\leq n/2$.
For those $\la$ of the length $\ell:=\ell(\la)\geq3$, we have $h\leq n-2$ and
\[\begin{split}
\ga_\la&=-2-\eta_\la=-2-(-1)^h(\eta_{\la-h}+(-1)^{\la_1}h\eta_{\la-1})\\
&\geq-2-|\eta_{\la-h}|-h|\eta_{\la-1}|\geq-2-D_{n-h}-hD_{n-\ell}\\
&>-2-(h+1)D_{n-\ell}\geq-2-(n-1)D_{n-\ell}\\
&>-2-nD_{n-3}-(-1)^n=\ga_{(n-2,1^2)}\qedhere
\end{split}\]
\end{proof}
We should remark that $\eta_{(n-2,2)}$ has been proved to be the second largest adjacency eigenvalue in \cite[Theorem 1.1]{DZ}, thus the result for $\ga_{(n-2,2)}$. Here we further compare $\ga_{(n-2,2)}$ with $\ga_{(n-1,1^2)}$. On the other hand, we have

\begin{theorem}
For $d_{\Ga_n}\,(n\geq4)$, the largest eigenvalue is \[\ga_{(n)}=2(n!-1)-D_n\]
with multiplicity one, the second largest eigenvalue is \[\ga_{(n-1,1)}=-2+\dfrac{D_n}{n-1}\]
with multiplicity $\chi_{(n-1,1)}(1)^2$, and the third largest eigenvalue is
\[\ga_{(n-3,3)}=-2+(-1)^{n-1}+\dfrac{(n-2)(n-3)}{n-5}D_{n-4}\]
with multiplicity $\chi_{(n-3,3)}(1)^2$ when $n\geq6$ and $\ga_{(2^2,1)}$ when $n=5$.
\end{theorem}

\begin{proof}
It is well-known that the largest adjacency eigenvalue for a $d$-regular graph is $d$. In particular for $d_{\Ga_n}$, the largest distance eigenvalue is $\ga_{(n)}=2(n!-1)-D_n$ with multiplicity one by Prop. \ref{di}. Besides, as $\eta_{(n-1,1)}$ is the minimum eigenvalue of $A_{\Ga_n}$ by Theorem 7.1 in \cite{Re}, the second largest eigenvalue for $d_{\Ga_n}$ is $\ga_{(n-1,1)}=-2+\tfrac{D_n}{n-1}$ with multiplicity $\chi_{(n-1,1)}(1)^2$.

Now for the third largest distance eigenvalue, we first know that
\[-\eta_{(n-3,3)}=|\eta_{(n-3,3)}|>|\eta_{(n-i,i)}|,\,\eta_{(n-3,3)}<-|\eta_{(n-i,i)}|\leq\eta_{(n-i,i)},
\,4\leq i\leq n/2\]
by Lemma \ref{l2}. Therefore, $\ga_{(n-3,3)}>\ga_{(n-i,i)}$ for $4\leq i\leq n/2$.
By Lemma \ref{l3}, Lemma\ref{l4},
\[\ga_{(n-3,3)}>\ga_{(n-i,1^i)},\ga_{(n-2-j,2,1^j)}\mb{ for }3\leq i\leq n-1,\,1\leq j\leq n-4,\,n\geq6.\]
By Lemma \ref{l5}, we also have $\ga_{(n-3,3)}> \ga_\la$ for any partition $\la$ of length three.
Thus it still needs to consider those $\la$ of length $\ell:=\ell(\la)\geq4$ and neither a hook nor a near hook. In this case, we have $h\leq n-2$ and
\[\begin{split}
\ga_\la&=-2-\eta_\la=-2-(-1)^h(\eta_{\la-h}+(-1)^{\la_1}h\eta_{\la-1})\\
&\leq-2+|\eta_{\la-h}|+h|\eta_{\la-1}|\leq-2+D_{n-h}+hD_{n-\ell}\\
&<-2+(h+1)D_{n-\ell}\leq-2+(n-1)D_{n-4}\\
&<-2+(-1)^{n-1}+\dfrac{(n-2)(n-3)}{n-5}D_{n-4}=\ga_{(n-3,3)}
\end{split}\]
In conclusion, we determine the third largest distance eigenvalue of $\Ga_n$.
\end{proof}

\begin{exam}
For the simplest example $n=4$, the distance polynomial of $\Ga_4$ is
\[P_{\Ga_4}(q)=(q-37)(q-1)^{10}(q+3)^9(q+5)^4,\]
where $\ga_{(4)}=37,\,\ga_{(3,1)}=\ga_{(1^4)}=1,\,\ga_{(2,1^2)}=-3,\,
\ga_{(2,2)}=-5$.

The distance polynomial of $\Ga_5$ is
\[P_{\Ga_5}(q)=(q-194)(q-9)^{16}(q-2)^{25}(q+1)^{16}(q+6)^{62}.\]
where $\ga_{(5)}=194,\,\ga_{(4,1)}=9,\,\ga_{(2^2,1)}=2,\,
\ga_{(2,1^3)}=-1,\,\ga_{(3,2)}=\ga_{(3,1^2)}=\ga_{(1^5)}=-6$.

The distance polynomial of $\Ga_6$ is
\[P_{\Ga_6}(q)=(q-1173)(q-51)^{25}(q-9)^{25}(q-3)^{357}
(q+3)^{25}(q+7)^{81}(q+9)^{25}(q+15)^{100}(q+17)^{81}.\]
where $\ga_{(6)}=1173,\,\ga_{(5,1)}=51,\,\ga_{(3^2)}=9,\,
\ga_{(3,2,1)}=\ga_{(3,1^3)}=\ga_{(1^6)}=3,\,
\ga_{(2,1^4)}=-1,\,\ga_{(2^2,1^2)}=-7,\,
\ga_{(2^3)}=-9,\,\ga_{(4,1^2)}=-15,\,\ga_{(4,2)}=-17$.
\end{exam}

\centerline{\bf Acknowledgments}
The first author was supported by NSFC (Grant No. 11501214) and the second author was supported by NSFC (Grant No. 11401211).

\bigskip
\bibliographystyle{amsalpha}

\begin{thebibliography}{9999}
\medskip

\bibitem{CK} P.J. Cameron, C.Y. Ku, \textit{Intersecting families of permutations}, European J. Combin. 24 (2003), 881--890.

\bibitem{CL} W.Y.C. Chen, J.D. Louck, \textit{The Factorial Schur Function}, J. Math. Phys. 34
(1993), 4144--4160.

\bibitem{DZ1} Y.-P. Deng, X.-D. Zhang, \textit{Automorphism group of the derangement graph}, Electron. J. Combin. 18 (2011), \#P198

\bibitem{DZ} Y.-P. Deng, X.-D. Zhang, \textit{A note on eigenvalues of the derangement graph}, Ars
Combin. 101 (2011), 289--299.

\bibitem{KW} C.Y. Ku, D.B. Wales, \textit{Eigenvalues of the derangement graph}, J. Combin. Theory Ser. A 117 (2010), 289--312.

\bibitem{KW1} C.Y. Ku, K.B. Wong, \textit{Solving the Ku--Wales conjecture on the eigenvalues of the derangement graph}, European J. Combin. 34 (2013), 941--956.

\bibitem{LM} B. Larose, C. Malvenuto, \textit{Stable sets of maximal size in Kneser-type graphs}, European J. Combin. 25 (2004), 657--673.

\bibitem{MST} K. Meagher, P. Spiga, T.H. Tiep, \textit{An Erd\H{o}s-Ko-Rado theorem for finite 2-transitive groups}, European J. Combin. 55 (2016), 100--118.

\bibitem{OO} A. Okounkov, G. Olshanski, \textit{Shifted Schur Functions}, Algebra i Analiz. 9 (1977),
73--146 (Russian). English translation: St. Petersburg Math. J. 9 (1998), 239--300.

\bibitem{Re} P. Renteln, \textit{On the spectrum of the derangement graph}, Electron. J. Combin. 14 (2007), \#R82.


\end{thebibliography}

\clearpage

\end{document}